\documentclass[10pt,twoside]{amsart}
\usepackage{amsopn,comment}
\usepackage{amsthm}
\usepackage{amssymb,amsmath,amscd, eurosym}
\usepackage{comment}
\usepackage{enumitem}
\usepackage[mathscr]{eucal}
\usepackage{subfig}
\usepackage{tikz}
\usepackage{xcolor}
\usepackage{float}

\DeclareMathOperator{\link}{link}
\DeclareMathOperator{\ass}{ass}

\newtheorem{theorem}{Theorem}[section]
\newtheorem{lemma}[theorem]{Lemma}
\newtheorem{proposition}[theorem]{Proposition}
\newtheorem{corollary}[theorem]{Corollary}

\newtheorem{construction}[theorem]{Construction}

\theoremstyle{definition}
\newtheorem{definition}[theorem]{Definition}
\theoremstyle{remark}
\newtheorem{remark}[theorem]{Remark}
\newtheorem{example}[theorem]{Example}

\newtheorem{question}[theorem]{Question}
\theoremstyle{plain}

\begin{document}
    \title[Codimension two ACM varieties in $\mathbb{P}^1\times\mathbb{P}^1\times\mathbb{P}^1$]{Special arrangements of lines: codimension two ACM varieties in $\mathbb{P}^1\times\mathbb{P}^1\times\mathbb{P}^1$}

\date{ \today}

    \author{Giuseppe Favacchio}
    \address{Dipartimento di Matematica e Informatica\\
        Viale A. Doria, 6 - 95100 - Catania, Italy}
    \email{favacchio@dmi.unict.it} \urladdr{http://www.dmi.unict.it/$\sim$gfavacchio/}

    \author{Elena Guardo}
    \address{Dipartimento di Matematica e Informatica\\
        Viale A. Doria, 6 - 95100 - Catania, Italy}
    \email{guardo@dmi.unict.it}
    \urladdr{http://www.dmi.unict.it/$\sim$guardo/}

    \author{Beatrice Picone}
    \address{Dipartimento di Matematica e Informatica\\
        Viale A. Doria, 6 - 95100 - Catania, Italy}
    \email{picone@dmi.unict.it}

\keywords{Configuration of lines, multiprojective spaces, arithmetically Cohen-Macaulay}
\subjclass[2010]{13C40, 13F20, 13A15, 14C20, 14M05}
\maketitle

\begin{abstract}In this paper we investigate special arrangements of lines in multiprojective spaces. In particular, we characterize codimensional two arithmetically Cohen-Macaulay (ACM)  varieties in $\mathbb P^1\times\mathbb P^1\times\mathbb P^1$, called  \textit{varieties of lines}.  We also describe their ACM property from combinatorial algebra point of view.
\end{abstract}
\section*{Introduction}
Given a variety $X\subseteq \mathbb P^{a_1} \times \dots \times \mathbb P^{a_n}$ an interesting problem is the description of the homological invariants of the coordinate ring of $X$.
This problem was especially studied for points and there is not a general answer in this direction. A great difficulty  came from the fact that a set of distinct points $X\subseteq \mathbb P^{a_1} \times \dots \times \mathbb P^{a_n}$ is not necessarily arithmetically Cohen-Macaulay (ACM). See, for instance \cite{FG2016a, GuVT2004,GuVT2008b,GuVT2012a,GV-book,GV15} for a short spectrum of the study on this topic,  and \cite{FGM, FM} for a recent characterization of the ACM property in $\mathbb P^1\times\cdots \times\mathbb P^1$ and, under certain conditions, in $\mathbb P^1\times\mathbb P^m.$
Recently, multiprojective spaces are getting more attention since many applications have been explored. For example, a specific value of the Hilbert function of a collection of (fat) points in a multiprojective space is related to a classical problem of algebraic geometry concerning the dimension of certain secant varieties of Segre
varieties (see \cite{BBC, CGG1, CGG2} just to cite some of them). Or in \cite{BBCG,CS} the readers could deduce new results about tensors, and in \cite{D}, the author focus on the implicitization problem for tensor product surfaces.

In particular, it appears of interest in combinatorial algebraic geometry  the study of  finite arrangements of lines  (see \cite{CHMN,H} for recent developments in $\mathbb{P}^2$).  A line arrangement over an algebraically closed field $K$ is a finite collection $L_1,\dots , L_d \subseteq \mathbb{P}^2$, $d > 1$, of distinct lines in the projective plane and their crossing points (i.e., the points of intersections of the lines).
In this paper we investigate special arrangements of lines in multiprojective spaces by focusing on ACM codimensional two varieties in $\mathbb{P}^1\times\mathbb P^1\times\mathbb P^1$, called \textit{varieties of lines}.  In particular, we study special cases arising from their intersection points (see Theorem \ref{thmACM}). These varieties can be viewed as special configurations of codimension two linear  varieties in $\mathbb P^5$.


The paper is structured as follows. In Section \ref{sec:def} we set up our notation and recall known results.  In Section \ref{sec:comb_char}, we describe a connection between ideals of varieties of lines and some squarefree monomial ideals (Lemma \ref{monomial}). We introduce in Definition  \ref{def:chordal} the $Hyp_n({\star})$-property to give a combinatorial characterization of ACM varieties of lines using a well known property of chordal graphs (Theorem \ref{thm:ch1}).
In Section \ref{sec:num_char} we introduce a numerical way to check the ACM property for any varieties of lines.
In Section \ref{sec:HF} we describe the Hilbert function of Ferrers varieties of lines, a special ACM case.
Finally, in Section \ref{sec:grids} we start an investigation on varieties of lines whose crossing points set is a complete
intersection of points in $\mathbb P^1\times\mathbb P^1\times\mathbb P^1$.
We also characterize varieties of lines defined by a complete intersection ideal
in $\mathbb P^1\times\mathbb P^1\times\mathbb P^1$  (Theorem \ref{thmcompleteintersection}).
  We end the paper with two possible research topics to explore: (1) the connection between our varieties of lines in  $\mathbb{P}^1 \times \mathbb{P}^1 \times \mathbb{P}^1$ and
  special configurations of lines of $\mathbb P^3$ and (2) the Hilbert function of any ACM variety of lines (Question
  \ref{HF}).

\noindent \textbf{Acknowledgement.} The authors thank A. Van Tuyl
for his comments in a previous version of the paper. They also
thank the referee for his/her helpful suggestion. The second
author would like to thank GNSAGA and Prin 2015.  The software
 CoCoA \cite{cocoa} was indispensable for all the computations.

\section{Notation and basic facts}\label{sec:def}

Throughout the paper $\mathbb{N} := \{0,1,2, \dots \}$ denotes the set of non-negative integers and $\preceq$ denotes the natural partial order on the elements of $\mathbb{N}^3 := \mathbb{N} \times \mathbb{N} \times \mathbb{N}$ defined by $(a_1,a_2,a_3) \preceq (b_1,b_2,b_3)$ in $\mathbb{N}^3$ if and only if $a_i \le b_i \ \forall i =1,2,3$.
Let $\{ \bf{e}_1,\bf{e}_2,\bf{e}_3 \}$ be the standard basis of $\mathbb N^3$. Let $R:=K[x_{1,0},x_{1,1},x_{2,0},x_{2,1},x_{3,0},x_{3,1}]$ be the polynomial ring  over an algebraically closed field $K$ of characteristic zero. We induce a multi-grading by setting $\deg x_{i,j}=\bf{e}_i$ for $i\in \{1,2,3\}$. A monomial $m=x_{1,0}^{a_0}x_{1,1}^{a_1}x_{2,0}^{b_0}x_{2,1}^{b_1}x_{3,0}^{c_0}x_{3,1}^{c_1} \in R$ has \textit{tridegree} (or simply, \textit{degree}) $\deg m= (a_0+a_1, b_0 + b_1, c_0 +c_1)$. We make the convention that $0$ has degree $\deg 0=(i,j,k)$ for all $(i,j,k) \in \mathbb{N}^3$. Note that the elements of the field $K$ all have degree $(0,0,0)$. For each $(i,j,k) \in \mathbb{N}^3$, let $R_{i,j,k}$ denote the vector space over $K$ spanned by all the monomials of $R$ of degree $(i,j,k)$. The polynomial ring $R$ is then a \textit{trigraded ring} because there exists a direct sum decomposition $$ R = \bigoplus \limits_{(i,j,k)\in \mathbb{N}^3} R_{i,j,k}$$ such that $R_{i,j,k}R_{l,m,n} \subseteq R_{i+l,j+m,k+n}$ for all $(i,j,k),(l,m,n) \in \mathbb{N}^3$. An element $F \in R$ is \textit{trihomogeneous} (or simply, \textit{homogeneous}) if $F \in R_{i,j,k}$ for some $(i,j,k) \in \mathbb{N}^3$. An ideal $I = (F_1,F_2, \dots ,F_r) \subseteq R$ is a \textit{(tri)homogeneous ideal} if $F_i$ is (tri)homogeneous for all $i=1,2, \dots ,r$.

Let $I \subseteq R$ be a homogeneous ideal, and we let $I_{i,j,k} := I \cap R_{i,j,k}$ for all $(i,j,k) \in \mathbb{N}^3$. Because $I$ is homogeneous, the quotient ring $R/I$ also inherists a graded ring structure. In particular, we have: $$  R/I = \bigoplus \limits_{(i,j,k)\in \mathbb{N}^3} (R/ I)_{i,j,k} =  \bigoplus \limits_{(i,j,k)\in \mathbb{N}^3} R_{i,j,k}/ I_{i,j,k}.$$

\begin{definition}
Let $I_X \subseteq R$ be the homogeneous ideal defining a variety $X \subseteq \mathbb{P}^1 \times \mathbb{P}^1 \times \mathbb{P}^1$. We say that $X$ is \textit{arithmetically Cohen-Macaulay} (ACM) if $R/I_X$ is Cohen-Macaulay, i.e. $\text{depth}(R/I_X) = \text{Krull-dim}(R/I_X)$.
\end{definition}
\begin{definition}
    We say that a homogeneous ideal $J$ in a polynomial ring $S$ is \textit{Cohen-Macaulay} (CM) if $S/J$ is Cohen-Macaulay.
\end{definition}

A point in $\mathbb P^1\times\mathbb P^1\times\mathbb P^1$ is an ordered set of three points in $\mathbb P^1$. Say $P:=([a_0,a_1],[b_0,b_1],[c_0,c_1])\in \mathbb P^1\times\mathbb P^1\times\mathbb P^1,$ the defining ideal of $P$ is
$I_P:=(a_1x_{1,0}-a_0x_{1,1},b_1x_{2,0}-b_0x_{2,1},c_1x_{3,0}-c_0x_{3,1})$. Note that $I_P$ is a height three prime ideal generated by homogeneous linear forms of different degree.

Throughout the paper, linear forms are denoted by capital letters. In particular, we use $A_i$ to denote a linear form of degree $(1,0,0)$,  $B_j$ a linear form of degree $(0,1,0)$, and $C_k$ a linear form of degree $(0,0,1).$ We denote by $\mathcal{L}(A_i)$, $\mathcal{L}(B_j)$ and $\mathcal{L}(C_k)$ the respective hyperplanes of  $\mathbb{P}^1\times\mathbb{P}^1\times\mathbb{P}^1$, and we say that a hyperplane in $\mathbb P^1\times\mathbb P^1\times\mathbb P^1$ is of type $\bf{e}_i$ if it is defined by a form of degree $\bf{e}_i$.

We recall the following definition (see \cite{GV15}, Definition 2.2).

\begin{definition}
Let $F, G \in R$ be two homogeneous linear forms of different degree. In $\mathbb P^1\times\mathbb P^1\times\mathbb P^1$ the variety $\mathcal{L}$ defined by the ideal $(F,G)\subseteq R$ is called a \textit{line} of $\mathbb P^1\times\mathbb P^1\times\mathbb P^1$ and we denote it by $\mathcal{L}(F,G)$.

We say that a line $\mathcal{L}(F,G)$ is of type $\bf{e}_i+\bf{e}_j$, with $i\neq j,$ if $\{\deg F, \deg G\}=\{\bf{e}_i, \bf{e}_j\}$.
\end{definition}

In particular, if $A\in R_{1,0,0}$, $B\in R_{0,1,0}$ and $C\in R_{0,0,1}$, then we denote by $\mathcal{L}(A,B)$ the variety in $\mathbb P^1\times\mathbb P^1\times\mathbb P^1$ defined by the ideal $(A,B)\subseteq R$ and we call it \textit{line of type $(1, 1, 0)$}. Analogously,  we call the variety $\mathcal{L}(A,C)$ \textit{line of type $(1, 0,1)$}  and the variety $\mathcal{L}(B,C)$ \textit{line of type $(0, 1, 1)$}. We also refer to lines of type $\bf{e}_1+\bf{e}_2$, $\bf{e}_1+\bf{e}_3$ and $\bf{e}_2+\bf{e}_3$ by writing \textit{lines having direction} $\bf{e_3}$, $\bf{e_2}$ and $\bf{e_1}$, respectively.

\begin{definition}
We say that $X \subseteq \mathbb P^1\times\mathbb P^1\times\mathbb P^1$ is a \textit{variety of lines} if it is given by a finite union of distinct lines in  $\mathbb P^1\times\mathbb P^1\times\mathbb P^1$.

\end{definition}

\begin{definition}\label{Def:U_i(X)}
Given $X\subseteq \mathbb{P}^1\times\mathbb{P}^1\times\mathbb{P}^1$ a variety of lines, we denote by
$\mathcal{H}_1(X):=\{\mathcal{L}(A_1),\ldots, \mathcal{L}(A_{d_1})\}$,   $ \mathcal{H}_2(X):=\{\mathcal{L}(B_1),\ldots, \mathcal{L}(B_{d_2})\}$   and   $\mathcal{H}_3(X):=\{\mathcal{L}(C_1),\ldots, \mathcal{L}(C_{d_3})\}$ the hyperplanes of $\mathbb{P}^1\times\mathbb{P}^1\times\mathbb{P}^1$  containing some lines of $X$. In particular:
$$X:= \bigcup_{(i,j)\in U_3(X)} \mathcal{L}(A_i,B_j)\cup\bigcup_{(i,k)\in U_2(X)} \mathcal{L}(A_i,C_k)\cup \bigcup_{(j,k)\in U_1(X)} \mathcal{L}(B_j,C_k)$$
where $U_3(X)\subseteq [d_1]\times[d_2], U_2(X)\subseteq [d_1]\times[d_3]$ and $U_1(X)\subseteq [d_2]\times[d_3]$ are sets of ordered pairs of integers, with $[n]:=\{1,2,\ldots, n\}\subset \mathbb{N}.$

For $i=1, 2, 3$, we denote by $X_i$ the {set of lines of $X$ having direction $\bf{e}_i$}
and we call $U_i(X)$ the \textit{ index set of $X_i$}.
\end{definition}

Thus, the ideal defining $X$ is
$$I_X=\bigcap_{(i,j)\in U_3(X)} (A_i,B_j)\bigcap_{(i,k)\in U_2(X)} (A_i,C_k)\bigcap_{(j,k)\in U_1(X)} (B_j,C_k).$$

In this paper we are interested in a combinatorial characterization of ACM  varieties of lines in $\mathbb{P}^1\times\mathbb{P}^1\times\mathbb{P}^1$ and their Hilbert function.
In $\mathbb P^1\times \mathbb P^1$, to describe combinatorially ACM sets of points, it was crucial the definition of the so called \textit{Ferrers diagram} (see for instance \cite{GV-book}).

\begin{definition}
    A tuple $\lambda = (\lambda_1 , \dots , \lambda_r)$ of positive integers is a \textit{partition} of an integer $s$ if \ $\sum_{i=1}^{r} \lambda_i = s$ and $\lambda_i \geq \lambda_{i+1}$ for every $i$. We write $\lambda = (\lambda_1 , \dots , \lambda_r) \vdash s$.
\end{definition}

\begin{definition}\label{d.Ferr}
    To any partition $\lambda = (\lambda_1 , \lambda_2, \dots , \lambda_r) \vdash s$ we can associate the following diagram: on an $r \times \lambda_1$ grid, place $\lambda_1$ points on the first horizontal line, $\lambda_2$ points on the second, and so on, where the points are left justified. The resulting diagram is called the \textit{Ferrers diagram} of the partition $\lambda$.
\end{definition}

\begin{definition} Let $Y$ be a finite set of points in $\mathbb{P}^1 \times \mathbb{P}^1$. We say that \textit{$Y$ resembles a Ferrers diagram} if the set of points looks like a Ferrers diagram, i.e., after relabeling the horizontal and vertical rulings, we can assume that the first horizontal ruling contains the most number of points of $Y$, the second contains the same number or less of points of $Y$, and so on.
\end{definition}

Applying Lemma 3.17, Theorem 3.21 and Theorem 4.11 in \cite{GV-book}, we have

\begin{lemma} \label{FD} Let $Y$ be a finite set of points in $\mathbb{P}^1 \times \mathbb{P}^1$.  $Y$ is ACM if and only if $Y$ resembles a Ferrers diagram.
\end{lemma}

We adapt Definition \ref{d.Ferr} to our context.
\begin{construction}  \label{construction} Let $X$ be a variety of lines in $\mathbb{P}^1 \times \mathbb{P}^1 \times \mathbb{P}^1$ and consider the set $X_3$ of lines of $X$ of type $(1,1,0)$ indexed by $U_3(X)\subseteq [d_1]\times [d_2]$. We represent $X_3$ as a $d_1 \times d_2$ grid, where the horizontal lines are labeled by the $\mathcal{L}(A_i)$'s for $i=1,\ldots, d_1$ and the vertical lines by the $\mathcal{L}(B_j)$'s for $j=1,\ldots, d_2$. By abuse of notation, we denote the horizontal lines by $\mathcal{L}(A_i)$ and the vertical lines by $\mathcal{L}(B_j)$.
Then, a line $\mathcal{L}(A_i,B_j) \in X_3$ is drawn as the intersection point of $\mathcal{L}(A_i)$ and $\mathcal{L}(B_j)$ in the grid. Similarly, we can construct a $d_1\times d_3$ grid  representing  $X_2$ and a $d_2\times d_3$ grid representing $X_1$.
\end{construction}

\begin{definition}\label{defferrer}
    Let $X$ be a variety of lines in $\mathbb{P}^1 \times \mathbb{P}^1 \times \mathbb{P}^1$ and $h\in\{1,2,3\}$. We say that
    $X$ \textit{resembles a Ferrers diagram with respect to the direction $h$} if the grid representing the lines of
    $X_h$, constructed as above,  resembles a Ferrers diagram.
\end{definition}

\begin{definition}\label{d.U Ferr}
    A finite subset $U = \{ (u_i,u_j) \} \subseteq \mathbb N^2$ \textit{resembles a Ferrers diagram} if it satisfies the following property: $$ (u_i,u_j) \in U \Rightarrow  (u_h,u_k) \in U \ \ \forall \ 1 \leq h \leq i \ , \ 1 \leq k \leq j. $$
\end{definition}

\begin{remark}\label{r. Ferr EQ}
    Note that Definition \ref{defferrer} is equivalent to say that the index set $U_h(X) \subset \mathbb{N}^2$ resembles a Ferrers diagram as Definition \ref{d.U Ferr}.
\end{remark}

\begin{remark}\label{r.cono} Construction \ref{construction} makes clear the connection between $X_h$ ($h\in \{1,2,3\}$) and a set
of points in $\mathbb{P}^1\times\mathbb{P}^1$.
$X_h$ is a cone of a set of distinct points on a hyperplane of $\mathbb{P}^1\times\mathbb{P}^1\times\mathbb{P}^1$.
So, we can look at it as a set of points in $\mathbb{P}^1\times\mathbb{P}^1$ with associated grid as described in the construction.  \end{remark}

\begin{example}\label{e.Ferr}
    Let $X$ be the following variety of 15 lines in $\mathbb{P}^1 \times \mathbb{P}^1 \times \mathbb{P}^1$:
    \begin{displaymath}
    \begin{split} X = & \mathcal{L}(A_1,B_2) \cup \mathcal{L}(A_1,B_4) \cup \mathcal{L}(A_1,B_5) \cup \mathcal{L}(A_2,B_2) \cup \mathcal{L}(A_2,B_3) \cup \\  & \cup\mathcal{L}(A_2,B_4) \cup \mathcal{L}(A_2,B_5) \cup \mathcal{L}(A_3,B_1) \cup \mathcal{L}(A_3,B_2) \cup \mathcal{L}(A_3,B_3) \cup \\ &\cup   \mathcal{L}(A_3,B_4)  \cup \mathcal{L}(A_3,B_5) \cup \mathcal{L}(A_4,B_4) \cup \mathcal{L}(B_1,C_1) \cup \mathcal{L}(B_2,C_2).
    \end{split}
    \end{displaymath}
    Then,
    \begin{displaymath}
    \begin{split}
    X_3= & \{\mathcal{L}(A_1,B_2), \mathcal{L}(A_1,B_4),  \mathcal{L}(A_1,B_5), \mathcal{L}(A_2,B_2),  \mathcal{L}(A_2,B_3),  \mathcal{L}(A_2,B_4),\mathcal{L}(A_2,B_5), \\ \phantom{=} &    \ \mathcal{L}(A_3,B_1),  \mathcal{L}(A_3,B_2),  \mathcal{L}(A_3,B_3),  \mathcal{L}(A_3,B_4),  \mathcal{L}(A_3,B_5),  \mathcal{L}(A_4,B_4)\}.
    \end{split}
    \end{displaymath}
    Using Construction \ref{construction},  $X_3$ is represented by a $4\times 5$ grid as Figure \ref{fig1}.
    After renaming, we see that $X_3$ resembles a Ferrers diagram of type $(5,4,3,1)$. Then, using Lemma \ref{FD}, $X_3$ is ACM (Figure \ref{fig2}).
      \begin{figure}[H]
        \begin{minipage}{.5\textwidth}
    \begin{tikzpicture}[scale=0.37]
            \draw (2,0.5) - - (2,9.5);
            \draw (4,0.5) - - (4,9.5);
            \draw (6,0.5) - - (6,9.5);
            \draw (8,0.5) - - (8,9.5);
            \draw (10,0.5) - - (10,9.5);

            \draw (0.5,2) - - (11.5,2);
            \draw (0.5,4) - - (11.5,4);
            \draw (0.5,6) - - (11.5,6);
            \draw (0.5,8) - - (11.5,8);
            \fill (8,2) circle (2mm);
            \fill (2,4) circle (2mm);
            \fill (10,6) circle (2mm);

            \fill (4,4) circle (2mm);
            \fill (4,6) circle (2mm);
            \fill (4,8) circle (2mm);
            \fill (6,4) circle (2mm);
            \fill (6,6) circle (2mm);
            \fill (6,4) circle (2mm);
            \fill (8,4) circle (2mm);
            \fill (10,4) circle (2mm);
            \fill (8,8) circle (2mm);
            \fill (10,8) circle (2mm);
            \fill (8,6) circle (2mm);

            \node [font=\footnotesize, text width=18mm] at (3,10) { $\mathcal{L}(B_1)$};

            \node [font=\footnotesize, text width=16mm] at (5,10) { $\mathcal{L}(B_2)$};

            \node [font=\footnotesize, text width=14mm] at (7,10) { $\mathcal{L}(B_3)$};

            \node [font=\footnotesize, text width=12mm] at (9,10) { $\mathcal{L}(B_4)$};

            \node [font=\footnotesize, text width=10mm] at (11,10) { $\mathcal{L}(B_5)$};

            \node [font=\footnotesize, text width=12mm] at (13.4,2) { $\mathcal{L}(A_4)$};

            \node [font=\footnotesize, text width=12mm] at (13.4,4) { $\mathcal{L}(A_3)$};

            \node [font=\footnotesize, text width=12mm] at (13.4,6) { $\mathcal{L}(A_2)$};

            \node [font=\footnotesize, text width=12mm] at (13.4,8) { $\mathcal{L}(A_1)$};
            \end{tikzpicture}
            \caption{ }  \label{fig1}
            \end{minipage}%
            \begin{minipage}{.5\textwidth}
            \begin{tikzpicture}[scale=0.37]

            \draw (2,0.5) - - (2,9.5);
            \draw (4,0.5) - - (4,9.5);
            \draw (6,0.5) - - (6,9.5);
            \draw (8,0.5) - - (8,9.5);
            \draw (10,0.5) - - (10,9.5);

            \draw (0.5,2) - - (11.5,2);
            \draw (0.5,4) - - (11.5,4);
            \draw (0.5,6) - - (11.5,6);
            \draw (0.5,8) - - (11.5,8);
            \fill (2,2) circle (2mm);
            \fill (2,4) circle (2mm);
            \fill (2,6) circle (2mm);
            \fill (2,8) circle (2mm);
            \fill (4,4) circle (2mm);
            \fill (4,6) circle (2mm);
            \fill (4,8) circle (2mm);
            \fill (6,4) circle (2mm);
            \fill (6,6) circle (2mm);
            \fill (6,8) circle (2mm);

            \fill (8,6) circle (2mm);
            \fill (8,8) circle (2mm);
            \fill (10,8) circle (2mm);

            \node [font=\footnotesize, text width=18mm] at (3,10) { $\mathcal{L}(\overline{B}_1)$};

            \node [font=\footnotesize, text width=16mm] at (5,10) { $\mathcal{L}(\overline{B}_2)$};

            \node [font=\footnotesize, text width=14mm] at (7,10) { $\mathcal{L}(\overline{B}_3)$};

            \node [font=\footnotesize, text width=12mm] at (9,10) { $\mathcal{L}(\overline{B}_4)$};

            \node [font=\footnotesize, text width=10mm] at (11,10) { $\mathcal{L}(\overline{B}_5)$};

            \node [font=\footnotesize, text width=12mm] at (13.4,2) { $\mathcal{L}(\overline{A}_4)$};

            \node [font=\footnotesize, text width=12mm] at (13.4,4) { $\mathcal{L}(\overline{A}_3)$};

            \node [font=\footnotesize, text width=12mm] at (13.4,6) { $\mathcal{L}(\overline{A}_2)$};

            \node [font=\footnotesize, text width=12mm] at (13.4,8) { $\mathcal{L}(\overline{A}_1)$};
            \end{tikzpicture}\caption{} \label{fig2}
        \end{minipage}
        \end{figure}
 \end{example}\begin{example}    Let $X$ be the variety of lines as in Example \ref{e.Ferr}.
 We have $X_1=  \{\mathcal{L}(B_1,C_1), \mathcal{L}(B_2,C_2)\}$ and  the $2\times 2$ grid representing
 $X_1$ does not resemble any Ferrers diagram (Figure \ref{fig3}). Thus $X$ does not resemble a Ferrers diagram with respect to the direction $1$.
 Hence, from Lemma \ref{FD}, $X_1$ is not ACM.
    \begin{figure}[H]
        \centering
        \begin{tikzpicture}[scale=0.35]

        \draw (2,4.5) - - (2,9.5);
        \draw (4,4.5) - - (4,9.5);

        \draw (0.5,6) - - (6,6);
        \draw (0.5,8) - - (6,8);
        \fill (2,8) circle (2mm);

        \fill (4,6) circle (2mm);

        \node [font=\footnotesize, text width=18mm] at (3,10) { $\mathcal{L}(B_1)$};

        \node [font=\footnotesize, text width=16mm] at (5,10) { $\mathcal{L}(B_2)$};




        \node [font=\footnotesize, text width=12mm] at (8,6) { $\mathcal{L}(A_2)$};

        \node [font=\footnotesize, text width=12mm] at (8,8) { $\mathcal{L}(A_1)$};
        \end{tikzpicture}
        \caption{} \label{fig3} \end{figure}
\end{example}

Since Ferrers diagrams play a crucial role in the characterization of the ACM property for a finite set of points in $\mathbb P^1\times \mathbb P^1$ (see for instance \cite{GV15}), it is natural for us to investigate the same property for a variety of lines  $X\subseteq \mathbb P^1\times \mathbb P^1\times \mathbb P^1$ (since $X$ has also codimension 2). In the next section, we will show that the ACM property of $X$ depends on the $X_i$ (see Corollary \ref{cor. Xi}), but the ACM-ness of the $X_i$ is not sufficient to ensure that $X$ is also ACM (see Remark \ref{rem. Xi}).

\section{A combinatorial characterization of ACM varieties of lines }\label{sec:comb_char}

In this section, we study the ACM property for varieties of lines from a combinatorial point of view. We refer to \cite{HH} for all the introductory material on monomial ideals.

The next lemma can be recovered from \cite{VT03}, Proposition 3.2.

\begin{lemma}\label{l.reg sequence}
    Let $X\subseteq \mathbb P^{1}\times\mathbb P^{1}\times\mathbb P^{1}$ be a variety of lines. Then, there exist three forms $A,B$ and $C$ of degree $(1,0,0), (0,1,0)$ and $(0,0,1)$, respectively, such that $(\bar A,\bar B,\bar C)$ is a regular sequence in $R/I_X.$
\end{lemma}
\begin{proof}
    Let $A\in R_{1,0,0}$ be such that $\mathcal{L}(A)\notin \mathcal H_1(X)$. We claim that $\bar A$ is a nonzero divisor of $R/I_X$. Indeed, take $F\in R$ a homogeneous form such that  $AF\in I_X$. Then $AF\in I_{\mathcal L}$, for any line $\mathcal L\in X$. Since $I_{\mathcal L}$ is a prime ideal and $A\notin I_{\mathcal L}$, then we get $F\in I_{\mathcal L}$, for any $\mathcal L\in X$.

    Now we prove the existence of the linear form $B.$ Since $X$ is ACM, then $J:=I_X+(A)$ is CM. Moreover, $J$ is homogeneous and its height is 3. Take the primary decomposition of $J$, say $J=\mathfrak q_1\cap \cdots \cap \mathfrak q_t $, and let $\mathfrak p_i:=\sqrt{\mathfrak q_i}$ for $i=1,\ldots, t.$
    The set of the nonzero divisors of $R/J$ is then  $\bigcup_{i}  \bar{ \mathfrak p}_i$.
    In order to prove that there exists an element $B\in R_{0,1,0}$ nonzero divisor of $R/J$, it is enough to show that $(\bigcup_{i}  { \mathfrak p}_i)_{0,1,0}\subsetneq R_{0,1,0}.$
    Since $R_{0,1,0}$ is a $K$-vector space over an infinite field, it is not a union of a finite number of its proper subspaces, and so it is enough to show that  $ (\mathfrak p_i)_{0,1,0}\subsetneq R_{0,1,0}$ for each $i=1,\ldots, t.$

    Let $i\in \{1,\ldots, t\}$, then we have  $I_X\subseteq J \subseteq \mathfrak p_i$. Therefore, there exists $\mathcal L\in X$ such that $I_X\subseteq I_{\mathcal L}\subseteq \mathfrak p_i$. This implies $\mathfrak p_i= I_{\mathcal L}+(A).$  Since $I_{\mathcal L}\neq R_{0,1,0}$ we are done.

    Analogously we prove the existence of a form $C\in R_{0,0,1}.$
\end{proof}

We set the notation for this section.
    Let $X$ be a variety of lines and $I_X$ its defining ideal
    $$I_X=\bigcap_{(i,j)\in U_3(X)} (A_i,B_j)\bigcap_{(i,k)\in U_2(X)} (A_i,C_k)\bigcap_{(j,k)\in U_1(X)} (B_j,C_k)\subseteq R.$$
    We construct a new polynomial ring in $d_1 + d_2 + d_3$ variables each of them corresponding to a hyperplane containing some lines of $X.$
    We denote by $S:=K[a_1, \ldots, a_{d_1}$, $b_1, \ldots, b_{d_2}, c_1, \ldots, c_{d_3}]$ the polynomial ring in $d_1 + d_2 + d_3$ variables and $\deg a_i=(1,0,0)$, $\deg b_j=(0,1,0)$, $\deg c_k=(0,0,1)$.
    We set
    $$J_X=\bigcap_{(i,j)\in U_3(X)} (a_i,b_j)\bigcap_{(i,k)\in U_2(X)} (a_i,c_k)\bigcap_{(j,k)\in U_1(X)} (b_j,c_k)\subseteq S.$$
    $J_X$ is a height 2 monomial ideal of $S$ and its associated primes correspond to the components of $X$.

The next lemma is crucial since, as its consequence, we can  connect homological invariants between  ACM varieties of lines and  some height 2 monomial ideals. Similar arguments were also used in \cite{FGM} (see proof of Theorem 3.2).

\begin{lemma}\label{monomial} Let $X$ be a variety of lines in $\mathbb{P}^1\times\mathbb{P}^1\times\mathbb{P}^1$. Then
 $X$ is ACM if and only if $J_X \subseteq S $ is CM.

\end{lemma}
\begin{proof}

    Set $T := S[x_{1,0}, x_{1,1}, x_{2,0}, x_{2,1}, x_{3,0}, x_{3,1}]$. Consider $J_X$ as an ideal, say $\overline{J}_X$, in the ring $T$. Since $J_X$ is a height 2 monomial ideal in $S$, then  $\overline{J}_X$, being a cone, continues to be a height 2 monomial ideal. Moreover, $\overline{J}_X$ has the same primary decomposition as $J_X$.
    Consider the linear forms $a_i-A_i$, $b_j-B_j$, $c_k-C_k$ and let $L$ be the ideal generated by all these linear forms.

    Assume $J_X$ is CM. Thus, in the quotient $T/(\overline{J}_X, L)$  we can view the
    addition of each linear form in $L$ as a proper hyperplane section.
    We have that $R/I_X$ and $T/(\overline{J}_X,L)$ both have height 2 and $ R/I_X\cong T/(\overline{J}_X, L)$. Then, since $J_X$ is CM, we get $X$ is ACM.

    On the other hand, if $X$ is ACM, then, applying Lemma \ref{l.reg sequence}, there exists a sequence of linear forms  $(A,B,C)\subseteq R$ that is regular in the quotient $R/I_X$.
    Let $\mathfrak q:= (A,B,C)\subseteq R$ be the ideal generated by these three linear forms. Consider the ideal $(I_X+\mathfrak q)/\mathfrak q \subseteq  R/\mathfrak q$, that can be viewed  as a codimension 2 monomial ideal in a polynomial ring in three variables.  Since a Hilbert-Burch matrix of $I_X$ has the same \lq\lq structure\rq\rq as the Hilbert-Burch matrix of a monomial ideal, i.e. it is a matrix with only two non zero entries in each column (see for instance Lemma 3.21 in \cite{FRZ} or Theorem 1.5. in \cite{Naeem}), then $I_X$ is generated by some products among the linear forms defining the lines of $X$. Since the addition of each linear form in $L$ can be seen as a proper hyperplane section, we also have $ R/I_X\cong T/(\overline{J}_X, L)$. Then $J_X$ is CM.

\end{proof}

\begin{corollary}\label{c.monomial} Let $X$ be an ACM variety of lines in $\mathbb{P}^1\times\mathbb{P}^1\times\mathbb{P}^1$. Then $I_X$ is generated by products of linear forms.

\end{corollary}

As a consequence of Lemma \ref{monomial}, it is interesting to further investigate the structure of the monomial ideal $J_X$ associated to $X$.
Now we recall only a few definitions we will use in the sequel.
We refer to \cite{HH,VT} for all preliminaries and for further results on graphs.

A (simple) graph $G$ is a pair $G =(V, E)$, where $V:=\{v_1,\ldots, v_N\}$ is a set of vertices of $G$ and $E$ is a collection of 2-subsets of $V,$ called the edges of $G$. The complementary graph of $G$, denoted by $G^c,$ is the graph $G^c = (V, E^c)$, where $E^c = \{\{v_i, v_j \}\ |\ \{v_i, v_j \} \notin E\}$.
 A sequence of vertices of $G$, $(v_1, v_2, \cdots, v_t)$, is a cycle of length $t$ if $\{v_1, v_2\},\{v_2, v_3\},\dots,\{v_t, v_1\}\in E$. A chord is an edge joining two not adjacent vertices in a cycle. A minimal cycle is a cycle without chords.
 A graph $G$ is called chordal when all its minimal cycles have length three.  We associate to a graph $G=(V,E)$ two squarefree monomial ideals in the ring $K[v_1, \ldots, v_N]$, the face ideal of $G$
$$I(G)=(v_iv_j | \{v_i,v_j\}\in E)$$
and the cover ideal of $G$
$$J(G)=\bigcap\limits_{\{v_i,v_j\}\in E }(v_i,v_j).$$
It is a well known fact that  $I(G)$ and $J(G)$ are the Alexander dual each of the other.
In the sequel we will use the following results:

\begin{theorem}[\cite{Fr}, Theorem 1] \label{th.Fr}
    Let $G$ be a graph. Then $I(G)$ has a linear resolution if and only if $G^c$ is a chordal graph.
\end{theorem}

\begin{theorem}[\cite{ER}, Theorem 3] \label{th.ER}
    Let $G$ be a graph. Then $I(G)$ has a linear resolution if and only if $J(G)$ is CM.
\end{theorem}

\begin{remark}\label{r.simpl compl} Let $X$ be a variety of lines. Let
    $G_X=(V_X, E_X)$ be the graph with vertex set $$V_X:=\{a_1, \ldots, a_{d_1}, b_1, \ldots, b_{d_2}, c_1, \ldots, c_{d_3} \}$$ and edge set
    \begin{displaymath}
    \begin{split}
    E_X := & \big\{ \{a_i,b_j\}\subseteq V_X\ |\ \mathcal{L}(A_i,B_j)\in X_3\big\} \cup \\
     \cup & \big\{ \{a_i,c_k\}\subseteq V_X\ |\ \mathcal{L}(A_i,C_k)\in X_2\big\}\cup \\
    \cup & \big\{ \{b_j,c_k\}\subseteq V_X\ |\ \mathcal{L}(B_j,C_k)\in X_1\big\}.
    \end{split}
    \end{displaymath}
    Then, we note that the monomial ideal $J_X$ is the cover ideal of the graph $G_X$: $$J_X = J(G_X)\subseteq S,$$ that is the Stanley-Reisner ideal of the simplicial complex (see Lemma 1.5.4. in \cite{HH})
    $$\Delta_X:= \langle V_X \setminus e\ |\ e \in  E_X \rangle.$$
\end{remark}

An useful application of Remark \ref{r.simpl compl} is the following lemma.

\begin{lemma}\label{remove-hyp}Let $X$ be an ACM variety of lines in $\mathbb{P}^1\times\mathbb{P}^1\times\mathbb{P}^1$ and let $\mathcal H\subseteq \mathbb{P}^1\times\mathbb{P}^1\times\mathbb{P}^1$ be a hyperplane containing some lines of $X$. Then the variety of lines $Y= \{ \mathcal{L} \in X \ | \ \mathcal{L} \not\subset \mathcal{H}\}$ is ACM.
\end{lemma}
\begin{proof}
    Let $H$ be the linear form defining $\mathcal H$. Denoted by $z$ the variable of $S$ corresponding to $H$ (the linear form $H$ is one of the forms $A_i, B_j, C_k$ and $z$ is the corresponding variable among $a_i, b_j, c_k$).
    We have
    \begin{itemize}
        \item[i)] $J_X:z= \bigcap\limits_{\tiny \begin{array}{c}
            \mathfrak p\in \ass(J_X)\\ z\notin \mathfrak p\\
            \end{array}}\!\!\!\mathfrak p.$ Both are monomial ideals, so the equality easily follows by checking the inclusions for monomials.  
        \item[ii)]  $J_X:z$ is the Stanley-Reisner ideal of the simplicial complex $\link_{\Delta_X} z$ (see sections 1.5.2 and 8.1.1 in \cite{HH}). Indeed, the Stanley-Reisner ideal of the $link$ of $z$ in $\Delta_X$ is generated by monomials corresponding to the elements $F\subseteq V_X$ such that  $\{z\}\cup F\notin \Delta_X$. All these monomials are in $J_X:z= I_{\Delta_X}:z$ and vice versa.
        \end{itemize}
    Then, in order to prove the statement, it is enough to show that $J_X:z$ is CM.
    From Lemma \ref{monomial}, we have that $J_X$ is CM, so the statement follows by Corollary 8.1.8 in \cite{HH}.
\end{proof}

\begin{corollary} \label{cor. Xi}
    If $X$ is an ACM variety of lines, then $X$ resembles a Ferrers diagram with respect to the direction $h$, for each $h = 1,2,3.$
\end{corollary}
\begin{proof}
    We show that $U_1(X)$ resembles a Ferrers diagram. Analogously, one can show the same for $U_2(X)$ and $U_3(X)$. Let us consider the variety of lines $X_1$ consisting of the lines of $X$ of type $(0,1,1).$  Since $I_{X_1}= I_{X\setminus\{\mathcal L (A_1 ), \ldots, \mathcal{L}(A_{d_1})  \} }$, $X_1$  preserves the ACM property by Lemma \ref{remove-hyp}. Moreover, $X_1=\bigcup\limits_{(j,k)\in U_1(X)}\mathcal{L}(B_j,C_k)$, i.e., it is a cone of an ACM set of distinct points on a hyperplane of $\mathbb{P}^1\times\mathbb{P}^1\times\mathbb{P}^1$, see Remark \ref{r.cono}.  A well known characterization, see for instance Theorem 4.11 in \cite{GV-book}, shows that this set of points resembles a Ferrers diagram. Using Remark \ref{r. Ferr EQ}, $U_1(X)$ resembles a Ferrers diagram.
    Then, the statement follows from Lemma \ref{FD}.

\end{proof}

\begin{remark} \label{rem. Xi}
From previous corollary, if there exists $i \in \{1,2,3\}$ such that $X_i$ is not ACM, then $X$ is not ACM. The following example shows that even if all $X_i$ are ACM $X$ could be not ACM.

\end{remark}

\begin{example} \label{ex not acm}
Let us consider the following variety of lines in $\mathbb{P}^1 \times \mathbb{P}^1 \times \mathbb{P}^1$:
        \begin{displaymath}
        \begin{split} X = \{ \mathcal{L}(A_1,B_1), \mathcal{L}(A_2,B_2), \mathcal{L}(B_3, C_3)\}.
        \end{split}
        \end{displaymath}
It is clear that the sets $X_1, X_2$ and $X_3$ resemble a Ferrers diagram, so each of them is ACM. But, in this case, $X$ is not ACM. This follows for instance from Lemma \ref{monomial} and from  \cite{HH}, Lemma 9.1.12.
\end{example}

    The next definition introduces a property for varieties of lines in  $\mathbb{P}^1 \times \mathbb{P}^1 \times \mathbb{P}^1$ in analogy to the known ($\star$)-property defined for sets of points in $\mathbb{P}^1 \times \mathbb{P}^1$ (see \cite{GuVT2012a}).

    \begin{definition}\label{def:star}
            Let $X\subseteq \mathbb{P}^1\times\mathbb{P}^1\times\mathbb{P}^1$ be a variety of lines. We say that $X$ has the \textit{($\star$)-property} (or explicitly, \textit{star property}) if given any two lines $L_1$, $ L_2 \in X$, there exists $L_3 \in X$ such that  $L_1, L_3$ and $L_2, L_3$ are coplanar.
    \end{definition}

    We slight generalize this property for varieties of lines.
    \begin{definition}\label{def:chordal}
        Let $X\subseteq \mathbb{P}^1\times\mathbb{P}^1\times\mathbb{P}^1$ be a variety of lines. Let $n \geq 4$, $n \in \mathbb{N}$, we say that $X$ has the \textit{$n$-hyperplanes ($\star$) property} (for short, $Hyp_n(\star)$-property) if given $n$ hyperplanes $H_1, H_2, \ldots, H_n$ such that $\mathcal{L}(H_i, H_j)\in X$ for any $j\neq i-1, i, i+1$ then $\mathcal{L}(H_u, H_{u+1})\in X$ for some $u\in \{ 1,2, \dots, n \}$, where $H_0 = H_n$ and $H_{n+1} = H_1$.
    \end{definition}

    \begin{remark} \label{n*}
    Note that if $n>6$, then $X$ has the $Hyp_n(\star)$-property. Indeed, among  $n>6$ hyperplanes there are at least three
            of the same type and so the condition $\mathcal{L}(H_i, H_j)\in X$ for any $j\neq i-1, i, i+1$ (where $H_0=H_n$ and $H_{n+1} =H_1$) fails to be true.

    \end{remark}
    \begin{remark}
        Note that the $Hyp_4(\star)$-property is equivalent to $(\star)$-property as Definition \ref{def:star}.
    \end{remark}

    \begin{example}
    Let us consider the following variety of lines in $\mathbb{P}^1 \times \mathbb{P}^1 \times \mathbb{P}^1$:
    \begin{displaymath} \begin{split} X = & \mathcal{L}(A_1,B_1) \cup \mathcal{L}(A_1,B_2) \cup \mathcal{L}(A_1,B_3) \cup \mathcal{L}(A_2,B_1) \cup \\ \cup & \mathcal{L}(A_2,B_2) \cup \mathcal{L}(A_1,C_1) \cup \mathcal{L}(A_1,C_2) \cup \mathcal{L}(A_2,C_1) \cup \\ \cup & \mathcal{L}(B_1,C_1) \cup \mathcal{L}(B_1,C_2) \cup \mathcal{L}(B_2,C_1) \cup \mathcal{L}(B_3,C_1). \end{split} \end{displaymath}
    \noindent $X$ has the $Hyp_4 (\star$)-property. Indeed, if we take the $4$ hyperplanes $\mathcal{L}(A_1)$, $ \mathcal{L}(A_2)$, $\mathcal{L}(B_1), \mathcal{L}(B_2)$, we have that $\mathcal{L}(A_1,B_1), \mathcal{L}(A_2,B_2) \in X$ and also $\mathcal{L}(A_1, B_2) \in X$; if we take the $4$ hyperplanes $\mathcal{L}(A_1), \mathcal{L}(A_2)$, $\mathcal{L}(B_1), \mathcal{L}(C_1)$, we have that $\mathcal{L}(A_1,B_1)$, $\mathcal{L}(A_2,C_1) \in X$ and also $\mathcal{L}(B_1, C_1) \in X$; and so on, if we take any two lines in $X$, there exists a third line in $X$ that is coplanar with the other two.
    \end{example}

    \begin{example}
        Let us consider the following variety of lines in $\mathbb{P}^1 \times \mathbb{P}^1 \times \mathbb{P}^1$:
        \begin{displaymath}
        \begin{split} X = & \mathcal{L}(A_1,B_1) \cup \mathcal{L}(A_1,B_2) \cup \mathcal{L}(A_1,B_3) \cup \mathcal{L}(A_2,B_2) \cup \mathcal{L}(A_1,C_1) \\ \cup &  \mathcal{L}(A_1,C_2) \cup \mathcal{L}(A_2,C_1) \cup \mathcal{L}(B_1,C_1) \cup \mathcal{L}(B_3,C_1).
        \end{split}
        \end{displaymath}
    \noindent $X$ has the $Hyp_5 (\star$)-property. Indeed, if we take the $5$ hyperplanes $\mathcal{L}(A_1)$, $ \mathcal{L}(A_2)$, $\mathcal{L}(B_1), \mathcal{L}(B_2)$, $\mathcal{L}(C_1)$ we have that the lines $\mathcal{L}(A_1,B_1), \mathcal{L}(A_1,B_2), \mathcal{L}(A_2,B_2)$, $\mathcal{L}(A_2,C_1)$, $\mathcal{L}(B_1,C_1) \in X$ and also $\mathcal{L}(A_1, C_1) \in X$; if we take the $5$ hyperplanes $\mathcal{L}(A_1), \mathcal{L}(A_2)$, $\mathcal{L}(B_3), \mathcal{L}(B_2),  \mathcal{L}(C_1)$, we have that $\mathcal{L}(A_1,B_3)$, $\mathcal{L}(A_1,B_2)$, $\mathcal{L}(A_2,B_2)$, $\mathcal{L}(A_2,C_1)$, $\mathcal{L}(B_3,C_1) \in X$ and also $\mathcal{L}(A_1, C_1) \in X$; and so on, if we take any 5 hyperplanes $H_1, \dots , H_5$ among $\mathcal{L}(A_1), \mathcal{L}(A_2), \mathcal{L}(B_1), \mathcal{L}(B_2), \mathcal{L}(B_3), \mathcal{L}(C_1), \mathcal{L}(C_2)$ such that $\mathcal{L}(H_i, H_j)\in X$ for any $j\neq i-1, i, i+1$, then there exists $u\in \{ 1, \dots , 5 \}$ such that $\mathcal{L}(H_u, H_{u+1})\in X$, where $H_0 = H_5$ and $H_{6} = H_1$. Note that if we take $\mathcal{L}(B_1), \mathcal{L}(B_2), \mathcal{L}(B_3)$ among the 5 hyperplanes we choose, the condition $\mathcal{L}(H_i, H_j)\in X$ for any $j\neq i-1, i, i+1$ fails to be true and then there is nothing to verify.
    \end{example}

The following theorem is the main result of this section.

\begin{theorem}\label{thm:ch1}
        Let $X$ be a variety of lines.  Then $X$ is ACM if and only if $X$ has the $Hyp_n(\star)$-property for $n= 4,5,6$.
\end{theorem}
\begin{proof}
Let $I_X$ be the ideal defining the variety of lines $X \subseteq \mathbb{P}^1 \times \mathbb{P}^1 \times \mathbb{P}^1$.
From Lemma \ref{monomial}, $X$ is ACM if and only if $J_X \subseteq S$ is CM. From Remark \ref{r.simpl compl}, the ideal $J_X$ is the cover ideal of the graph $G_X$, i.e., $J_X = J(G_X)$. From Theorem \ref{th.ER}, the face ideal $I(G_X)$ has a linear resolution and then, using Theorem \ref{th.Fr}, $G_X^c$ is a chordal graph, that is, $X$ has the $Hyp_n(\star)$-property for any $n$. Remark \ref{n*} completes the proof.
\end{proof}

\section{A numerical characterization of the ACM property}\label{sec:num_char}

   Since we are interested in the study of the ACM property for varieties of lines $X$, from now on we assume that $U_h(X)$ resembles a Ferrers diagram for each $h=1,2,3.$
   In order to give a characterization of the ACM property we introduce the following notation.

    \begin{definition}\label{def:mi}
    Let $P=P_{ijk}=\mathcal{L}(A_i)\cap\mathcal{L}(B_j)\cap\mathcal{L}(C_k)$ be a point of a variety of lines $X$, we call \textit{multiplicity of $P$}  the number of lines of $X$ passing through the point $P$ and we denote it by $\mu_{ijk}.$
    \end{definition}

    \begin{remark}
    Since at most three lines of $X$ (one of each type) pass through the point $P$, $\mu_{ijk} \leq 3$.
    \end{remark}

 \begin{definition} Given a variety of lines $X$, we define  a 3-dimensional matrix  $M_X:=(\mu_{ijk})\in \mathbb{N}^{d_1\times d_2\times d_3}$ whose $(i,j,k)$-entry is the multiplicity of  $P_{i,j,k}.$  We call it the \textit{matrix of the multiplicities} of $X$.
 \end{definition}

We also define
   \begin{definition}
   $M_X^{({3})}:=(\mu_{ij0})\in \mathbb{N}^{d_1\times d_2}$, where
   $$\mu_{ij0}:=\begin{cases}
    1 & \ \text{if}\ (i,j) \in U_3(X),\ \text{i.e.,}\  \mathcal{L}(A_{i},B_{j})\in X\\
    0 & \ \text{otherwise}.
   \end{cases}$$

   \noindent Analogously, $M_X^{(2)}:=(\mu_{i0 k})\in \mathbb{N}^{d_1\times d_3}$, where $\mu_{i0k}:=\begin{cases}
    1 & \ \text{if}\ \mathcal{L}(A_{i},C_{k})\in X\\
    0 & \ \text{otherwise}
      \end{cases}$

   \noindent and $M_X^{({1})}:=(\mu_{0 jk})\in \mathbb{N}^{d_2\times d_3}$, where $\mu_{0jk}:=\begin{cases}
    1 & \ \text{if}\ \mathcal{L}(B_j, C_k)\in X\\
    0 & \ \text{otherwise}
      \end{cases}.$
    \end{definition}

\begin{example} \label{examplemultiplicity}
Let us consider $X \subseteq \mathbb{P}^1 \times \mathbb{P}^1 \times \mathbb{P}^1$ as Figure \ref{figure6starproperty}
\begin{displaymath}
\begin{split} X = & \mathcal{L}(A_1,B_1) \cup \mathcal{L}(A_1,B_2) \cup \mathcal{L}(A_2,B_2) \cup \mathcal{L}(A_1,C_1) \cup \mathcal{L}(A_2,C_1) \\ \cup & \mathcal{L}(A_2,C_2) \cup \mathcal{L}(B_1,C_1) \cup \mathcal{L}(B_1,C_2) \cup \mathcal{L}(B_2,C_2).
\end{split}
\end{displaymath}
We have
 $$ \mu_{111} = \mu_{222} = 3 \ \ , \ \  \mu_{121} = \mu_{221} = \mu_{211} = \mu_{112} = \mu_{122} = \mu_{212} = 2 $$
and  \begin{displaymath} M_X^{(3)} = \left( \begin{matrix} 1 & 1 \\ 0 & 1
 \end{matrix} \right), \ \  M_X^{(2)} = \left( \begin{matrix} 1 & 0 \\ 1 & 1
 \end{matrix} \right), \ \  M_X^{(1)} = \left( \begin{matrix} 1 & 1 \\ 0 & 1
 \end{matrix} \right).
 \end{displaymath}
  \begin{center}
   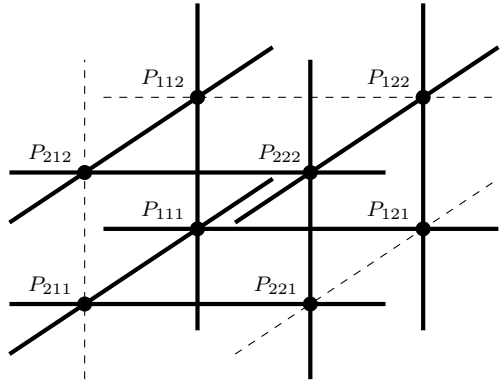
\begin{figure}[H]
   \centering
   \begin{tikzpicture}[scale=0.5]

   \draw[dashed] (2,0) - - (2,8.5);
   \draw [ultra thick] (5,1.3) - - (5,10);
   \draw [ultra thick](8,0) - - (8,8.5);
   \draw [ultra thick] (11,1.3) - - (11,10);

   \draw [ultra thick] (0,2) - - (10,2);
   \draw [ultra thick] (2.5,4) - - (13,4);
   \draw [ultra thick] (0,55/10) - - (10,55/10);
   \draw[dashed] (2.5,75/10) - - (13,75/10);

   \draw [ultra thick] (0,2/3) - - (7,16/3);
   \draw [ultra thick] (0,25/6) - - (7,53/6);
   \draw[dashed] (6,2/3) - - (13,16/3);
   \draw [ultra thick] (6,25/6) - - (13,53/6);

   \fill (2,2) circle (2mm);
   \fill (8,2) circle (2mm);
   \fill (8,5.5) circle (2mm);
   \fill (2,5.5) circle (2mm);
   \fill (5,4) circle (2mm);
   \fill (11,4) circle (2mm);
   \fill (11,7.5) circle (2mm);
   \fill (5,7.5) circle (2mm);

   \node [font=\footnotesize, text width=10mm] at (4.5,4.5) { $P_{111}$};
    \node [font=\footnotesize, text width=10mm] at (10.5,4.5) { $P_{121}$};
   \node [font=\footnotesize, text width=10mm] at (1.5,2.5) { $P_{211}$};
   \node [font=\footnotesize, text width=10mm] at (7.5,2.5) { $P_{221}$};
   \node [font=\footnotesize, text width=10mm] at (4.5,8) { $P_{112}$};
   \node [font=\footnotesize, text width=10mm] at (10.5,8) { $P_{122}$};
   \node [font=\footnotesize, text width=10mm] at (1.5,6) { $P_{212}$};
   \node [font=\footnotesize, text width=10mm] at (7.6,6) { $P_{222}$};

   \end{tikzpicture} \caption{The variety of lines $X$ (in bold).} \label{figure6starproperty}
   \end{figure}
   \end{center}

\end{example}

   Now we provide a criterion to establish if $X$ is ACM or not just looking at the matrices of the multiplicities $M_X$, $M_X^{(1)}$, $M_X^{(2)}$ and $M_X^{(3)}$.

   \begin{proposition} \label{6starproperty}
    Let $X$ be a variety of lines. $X$ has the $Hyp_6(\star)$-property iff for all $a_1,a_2\in[d_1],b_1,b_2\in[d_2],c_1,c_2\in[d_3]$
     $$\text{either }\  \left(\begin{array}{cc}
     \mu_{a_1b_1c_1} & \mu_{a_1b_2c_1}\\
     \mu_{a_2b_1c_1} & \mu_{a_2b_2c_1}\\
     \end{array}\right)\neq\left(\begin{array}{cc}
     3 & 2\\
     2 & 2\\
     \end{array}\right)\ \text{or} \ \left(\begin{array}{cc}
     \mu_{a_1b_1c_2} & \mu_{a_1b_2c_2}\\
     \mu_{a_2b_1c_2} & \mu_{a_2b_2c_2}\\
     \end{array}\right)\neq \left(\begin{array}{cc}
     2 & 2\\
     2 & 3\\
     \end{array}\right).$$
   \end{proposition}
   \begin{proof} If $X$ does not have the $Hyp_6(\star)$-property then there exist six planes, say $\mathcal{L}(A_1), \mathcal{L}(A_2), \mathcal{L}(B_1), \mathcal{L}(B_2), \mathcal{L}(C_1), \mathcal{L}(C_2)$, such that the lines $\mathcal{L}(A_1,B_1),$ $\mathcal{L}(A_1,B_2),$ $\mathcal{L}(A_1,C_1), $ $\mathcal{L}(A_2,B_2),$ $\mathcal{L}(A_2,C_1),$ $\mathcal{L}(A_2,C_2), $ $\mathcal{L}(B_1,C_1), $ $\mathcal{L}(B_1,C_2), $ $\mathcal{L}(B_2,C_2)$ belong to $X$    and $\mathcal{L}(A_2,B_1), $ $\mathcal{L}(B_2,C_1), $ $\mathcal{L}(A_1,C_2)\notin X.$ Then  we have that
        $$\left(\begin{array}{cc}
        \mu_{111} & \mu_{121}\\
        \mu_{211} & \mu_{221}\\
        \end{array}\right)=\left(\begin{array}{cc}
        3 & 2\\
        2 & 2\\
        \end{array}\right)\ \text{and} \ \left(\begin{array}{cc}
        \mu_{112} & \mu_{122}\\
        \mu_{212} & \mu_{222}\\
        \end{array}\right)= \left(\begin{array}{cc}
        2 & 2\\
        2 & 3\\
        \end{array}\right).$$
        On the other hand if    $\left(\begin{array}{cc}
        \mu_{111} & \mu_{121}\\
        \mu_{211} & \mu_{221}\\
        \end{array}\right)=\left(\begin{array}{cc}
        3 & 2\\
        2 & 2\\
        \end{array}\right)\ \text{and} \ \left(\begin{array}{cc}
        \mu_{112} & \mu_{122}\\
        \mu_{212} & \mu_{222}\\
        \end{array}\right)= \left(\begin{array}{cc}
        2 & 2\\
        2 & 3\\
        \end{array}\right)$ then it is easy to check that $X$ does not have the $Hyp_6(\star)$-property since the lines
        $\mathcal{L}(A_1,B_1),$ $\mathcal{L}(A_1,B_2),$ $\mathcal{L}(A_1,C_1), $ $\mathcal{L}(A_2,B_2),$ $\mathcal{L}(A_2,C_1),$ $\mathcal{L}(A_2,C_2), $ $\mathcal{L}(B_1,C_1), $ $\mathcal{L}(B_1,C_2), $ $\mathcal{L}(B_2,C_2)\in X $   and $\mathcal{L}(A_2,B_1), $ $\mathcal{L}(B_2,C_1), $ $\mathcal{L}(A_1,C_2)\notin X.$
     \end{proof}

     \begin{proposition} \label{5starproperty}
        Let $X$ be a variety of lines. $X$ has the $Hyp_5(\star)$-property iff for all $a_1,a_2\in[d_1],b_1,b_2\in[d_2],c_1,c_2\in[d_3]$ the following three conditions hold:
        $$1) \text{either}\left(\begin{array}{cc}
        \mu_{a_1b_1c_1} & \mu_{a_1b_2c_1}\\
        \mu_{a_2b_1c_1} & \mu_{a_2b_2c_1}\\
        \end{array}\right)\neq\left(\begin{array}{cc}
        2 & 1\\
        2 & 2\\
        \end{array}\right)\ \text{or} \ \left(\begin{array}{cc}
        \mu_{a_1b_10} & \mu_{a_1b_20}\\
        \mu_{a_2b_10} & \mu_{a_2b_20}\\
        \end{array}\right)\neq \left(\begin{array}{cc}
        1 & 1\\
        0 & 1\\
        \end{array}\right),$$

        $$2) \text{either}\left(\begin{array}{cc}
                \mu_{a_1b_1c_1} & \mu_{a_1b_1c_2}\\
                \mu_{a_2b_1c_1} & \mu_{a_2b_1c_2}\\
                \end{array}\right)\neq\left(\begin{array}{cc}
                2 & 1\\
                2 & 2\\
                \end{array}\right)\ \text{or} \ \left(\begin{array}{cc}
                \mu_{a_10 c_1} & \mu_{a_10 c_2}\\
                \mu_{a_20 c_1} & \mu_{a_20 c_2}\\
                \end{array}\right)\neq \left(\begin{array}{cc}
                1 & 1\\
                0 & 1\\
                \end{array}\right),$$

         $$3) \text{either}\left(\begin{array}{cc}
                                \mu_{a_1b_1c_1} & \mu_{a_1b_1c_2}\\
                                \mu_{a_1b_2c_1} & \mu_{a_1b_2c_2}\\
                                \end{array}\right)\neq\left(\begin{array}{cc}
                                2 & 1\\
                                2 & 2\\
                                \end{array}\right)\ \text{or} \ \left(\begin{array}{cc}
                                \mu_{0 b_1c_1} & \mu_{0 b_1c_2}\\
                                \mu_{0 b_2c_1} & \mu_{0 b_2c_2}\\
                                \end{array}\right)\neq \left(\begin{array}{cc}
                                1 & 1\\
                                0 & 1\\
                                \end{array}\right).$$

      \end{proposition}
      \begin{proof}





\noindent If $X$ does not have the $Hyp_5(\star)$-property, we say, without loss of generality, that exist five planes $\mathcal{L}(A_1), \mathcal{L}(A_2), \mathcal{L}(B_1), \mathcal{L}(B_2), \mathcal{L}(C_1)$ such that, among all, only the lines $\mathcal L(A_2,B_1)$, \ $\mathcal L(A_1,C_1)$, \ $\mathcal L(B_2,C_1) \notin X.$
               Then we have

               $\left(\begin{array}{cc}
               \mu_{111} & \mu_{121}\\
               \mu_{211} & \mu_{221}\\
               \end{array}\right)=\left(\begin{array}{cc}
               2 & 1\\
               2 & 2\\
               \end{array}\right)\ \ \text{and}  \ \left(\begin{array}{cc}
               \mu_{110} & \mu_{120}\\
               \mu_{210} & \mu_{220}\\
               \end{array}\right)= \left(\begin{array}{cc}
               1 & 1\\
               0 & 1\\
               \end{array}\right).$

\noindent On the other hand,  assume, for instance, we have the following equalities

                $\left(\begin{array}{cc}
                \mu_{111} & \mu_{121}\\
                \mu_{211} & \mu_{221}\\
                \end{array}\right)=\left(\begin{array}{cc}
                2 & 1\\
                2 & 2\\
                \end{array}\right)\ \ \text{and}  \ \left(\begin{array}{cc}
                \mu_{110} & \mu_{120}\\
                \mu_{210} & \mu_{220}\\
                \end{array}\right)= \left(\begin{array}{cc}
                1 & 1\\
                0 & 1\\
                \end{array}\right).$

                 \noindent From the previous equalities, we get $\mathcal{L}(A_1,C_1)\notin X$ thus, since $\mu_{111}=2,$ we have $\mathcal{L}(A_1,B_1),\mathcal{L}(B_1,C_1)\in X.$ Analogously $\mathcal{L}(A_1,B_2)\in X$ and so, since $\mu_{121}=1,$ we have $\mathcal{L}(B_2,C_1)\notin X.$
                 Moreover $\mathcal{L}(A_2,B_2)\in X$ and so, since $\mu_{221} = 2$, we have $\mathcal L(A_2,C_1) \in X$. Finally $\mathcal L(A_2,B_1) \notin X$ since $\mu_{210}=0$.  So $X$ does not have the $Hyp_5(\star)$-property.
         \end{proof}

      \begin{proposition} \label{4starproperty}
        Let $X$ be a variety of lines. $X$ has the $Hyp_4(\star)$-property iff for all  $a_1,a_2\in[d_1],b_1,b_2\in[d_2],c_1,c_2\in[d_3]$ the following three conditions hold:
        $$1) \text{either} \ \left(\begin{array}{c}
        \mu_{a_1b_1c_1} \\
        \mu_{a_2b_1c_1} \\
        \end{array}\right)\neq\left(\begin{array}{cc}
        1 \\
        1 \\
        \end{array}\right)\ \text{or} \ \left(\begin{array}{cc}
        \mu_{a_1b_10}\\
        \mu_{a_2b_10}\\
        \end{array}\right)\neq \left(\begin{array}{cc}
        1 \\
        0 \\
        \end{array}\right).$$

        $$2) \text{either} \ \left(\begin{array}{c}
                \mu_{a_1b_1c_1} \\
                \mu_{a_1b_1c_2} \\
                \end{array}\right)\neq\left(\begin{array}{cc}
                1 \\
                1 \\
                \end{array}\right)\ \text{or} \ \left(\begin{array}{cc}
                \mu_{a_10c_1}\\
                \mu_{a_10c_2}\\
                \end{array}\right)\neq \left(\begin{array}{cc}
                1 \\
                0 \\
                \end{array}\right).$$

        $$3) \text{either} \ \left(\begin{array}{c}
                \mu_{a_1b_1c_1} \\
                \mu_{a_1b_2c_1} \\
                \end{array}\right)\neq\left(\begin{array}{cc}
                1 \\
                1 \\
                \end{array}\right)\ \text{or} \ \left(\begin{array}{cc}
                \mu_{0b_1c_1}\\
                \mu_{0b_2c_1}\\
                \end{array}\right)\neq \left(\begin{array}{cc}
                1 \\
                0 \\
                \end{array}\right).$$
      \end{proposition}
      \begin{proof}
      Suppose that $X$ does not have the $Hyp_4(\star)$-property.
       Since we are assuming do not exist four planes $\mathcal{L}(A_1), \mathcal{L}(A_2), \mathcal{L}(B_1), \mathcal{L}(B_2)$ such that $\mathcal L(A_1,B_1), \mathcal L(A_2,B_2) \in X$ and $\mathcal L(A_1, B_2)$ or $\mathcal L(A_2,B_1) \notin X$, then,  $X$ fails the $Hyp_4(\star)$-property if, without loss of generality, there exist four planes $\mathcal{L}(A_1), \mathcal{L}(A_2), \mathcal{L}(B_1), \mathcal{L}(C_1)$ such that, among all, only the lines $\mathcal L(A_2,B_1),\mathcal L(B_1,C_1), \mathcal L(A_1,C_1)\notin X.$
            Then we have

           $\left(\begin{array}{cc}
            \mu_{111} \\
            \mu_{211} \\
            \end{array}\right)=\left(\begin{array}{cc}
            1 \\
            1 \\
            \end{array}\right)\ \ \text{and}   \left(\begin{array}{cc}
            \mu_{110} \\
            \mu_{210} \\
            \end{array}\right)= \left(\begin{array}{cc}
            1 \\
            0 \\
            \end{array}\right).$

\noindent On the other hand,  assume, for instance, we have the following equalities

                $\left(\begin{array}{cc}
                \mu_{111} \\
                \mu_{211} \\
                \end{array}\right)=\left(\begin{array}{cc}
                1 \\
                1 \\
                \end{array}\right)\ \ \text{and}  \ \left(\begin{array}{cc}
                \mu_{110} \\
                \mu_{210} \\
                \end{array}\right)= \left(\begin{array}{cc}
                1 \\
                0 \\
                \end{array}\right).$

                 \noindent From the previous equalities, we get $\mathcal{L}(A_1,B_1)\in X$ thus, since $\mu_{111}=1,$ we have $\mathcal{L}(A_1,C_1),\mathcal{L}(B_1,C_1)\notin X.$ Analogously $\mathcal{L}(A_2,B_1)\notin X$ and so, since $\mu_{211}=1,$ we have $\mathcal{L}(A_2,C_1)\in X$. So $X$ does not have the $Hyp_4(\star)$-property.
     \end{proof}
\begin{example}
Let $X$ be as in Example \ref{examplemultiplicity} (see Figure \ref{figure6starproperty}). We observe that

$\left(\begin{array}{cc}
     \mu_{111} & \mu_{121}\\
     \mu_{211} & \mu_{221}\\
     \end{array}\right)= \left(\begin{array}{cc}
     3 & 2\\
     2 & 2\\
     \end{array}\right)  \text{and} \ \left(\begin{array}{cc}
            \mu_{112} & \mu_{122}\\
            \mu_{212} & \mu_{222}\\
            \end{array}\right)= \left(\begin{array}{cc}
            2 & 2\\
            2 & 3\\
            \end{array}\right)$

\noindent and then, by Proposition \ref{6starproperty}, we have that $X$ does not have the $Hyp_6(\star)$-property and so, by Theorem \ref{thm:ch1}, $X$ is not ACM.
\end{example}

\begin{example}

\noindent Let us consider the variety $W=X \cup \mathcal{L}(A_2,B_1)$, where $X$ is as Example \ref{examplemultiplicity}.

 \begin{center}
   \begin{figure}[H]
   \centering
   \begin{tikzpicture}[scale=0.48]





   \draw [ultra thick] (2,0) - - (2,8.5);
   \draw [ultra thick] (5,1.3) - - (5,10);
   \draw [ultra thick](8,0) - - (8,8.5);
   \draw [ultra thick] (11,1.3) - - (11,10);

   \draw [ultra thick] (0,2) - - (10,2);
   \draw [ultra thick] (2.5,4) - - (13,4);
   \draw [ultra thick] (0,55/10) - - (10,55/10);
   \draw[dashed](2.5,75/10) - - (13,75/10);

   \draw [ultra thick] (0,2/3) - - (7,16/3);
   \draw [ultra thick] (0,25/6) - - (7,53/6);
   \draw[dashed] (6,2/3) - - (13,16/3);
   \draw [ultra thick] (6,25/6) - - (13,53/6);

   \fill (2,2) circle (2mm);
   \fill (8,2) circle (2mm);
   \fill (8,5.5) circle (2mm);
   \fill (2,5.5) circle (2mm);
   \fill (5,4) circle (2mm);
   \fill (11,4) circle (2mm);
   \fill (11,7.5) circle (2mm);
   \fill (5,7.5) circle (2mm);

   \node [font=\footnotesize, text width=10mm] at (4.5,4.5) { $P_{111}$};
    \node [font=\footnotesize, text width=10mm] at (10.5,4.5) { $P_{121}$};
   \node [font=\footnotesize, text width=10mm] at (1.5,2.5) { $P_{211}$};
   \node [font=\footnotesize, text width=10mm] at (7.5,2.5) { $P_{221}$};
   \node [font=\footnotesize, text width=10mm] at (4.5,8) { $P_{112}$};
   \node [font=\footnotesize, text width=10mm] at (10.5,8) { $P_{122}$};
   \node [font=\footnotesize, text width=10mm] at (1.5,6) { $P_{212}$};
   \node [font=\footnotesize, text width=10mm] at (7.6,6) { $P_{222}$};


   \end{tikzpicture}\caption{The variety of lines $W$ (in bold).} \label{figure6starproperty2}
   \end{figure}
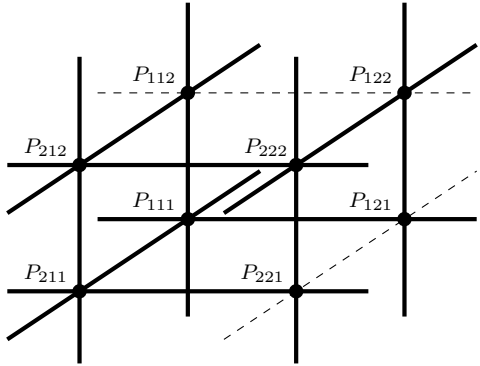
   \end{center}
We have $ \mu_{111} = \mu_{222} = \mu_{211} = \mu_{212} =3, \ \mu_{121} = \mu_{221} = \mu_{112} = \mu_{122} = 2. $
And for all  $ a_1,a_2\in[2],\ b_1,b_2\in[2],\ c_1,c_2 \in[2]$, we have:
$$\left(\begin{array}{cc}
     \mu_{a_1b_1c_1} & \mu_{a_1b_2c_1}\\
     \mu_{a_2b_1c_1} & \mu_{a_2b_2c_1}\\
     \end{array}\right)\neq  \left(\begin{array}{cc}
     3 & 2\\
     2 & 2\\
     \end{array}\right),
\left(\begin{array}{cc}
    \mu_{a_1b_1c_1} & \mu_{a_1b_2c_1}\\
    \mu_{a_2b_1c_1} & \mu_{a_2b_2c_1}\\
    \end{array}\right)\neq  \left(\begin{array}{cc}
    2 & 1\\
    2 & 2\\
    \end{array}\right), $$

$$    \left(\begin{array}{cc}
                    \mu_{a_1b_1c_1} & \mu_{a_1b_1c_2}\\
                    \mu_{a_2b_1c_1} & \mu_{a_2b_1c_2}\\
                    \end{array}\right)\neq  \left(\begin{array}{cc}
                    2 & 1\\
                    2 & 2\\
                    \end{array}\right), \left(\begin{array}{cc}
                        \mu_{a_1b_1c_1} & \mu_{a_1b_1c_2}\\
                        \mu_{a_1b_2c_1} & \mu_{a_1b_2c_2}\\
                        \end{array}\right)\neq  \left(\begin{array}{cc} 2 & 1\\
                          2 & 2\\
                        \end{array}\right), $$

$$ \left(\begin{array}{c}
        \mu_{a_1b_1c_1} \\
        \mu_{a_2b_1c_1} \\
        \end{array}\right)\neq   \left(\begin{array}{cc}
        1 \\
        1 \\
        \end{array}\right), \left(\begin{array}{c}
                \mu_{a_1b_1c_1} \\
                \mu_{a_1b_1c_2} \\
                \end{array}\right)\neq  \left(\begin{array}{cc}
                1 \\
                1 \\
                \end{array}\right), \left(\begin{array}{c}
                        \mu_{a_1b_1c_1} \\
                        \mu_{a_1b_2c_1} \\
                        \end{array}\right)\neq  \left(\begin{array}{cc}
                        1 \\
                        1 \\
                        \end{array}\right) $$
 and then, by Propositions \ref{6starproperty}, \ref{5starproperty} and \ref{4starproperty}, the variety of lines $W$ has the $Hyp_n(\star)$-property for $n=4, 5, 6$ and then, by Theorem \ref{thm:ch1}, $W$ is ACM.
\end{example}

\section{The Hilbert function of ACM codimension two varieties in $\mathbb P^1 \times \mathbb P^1 \times \mathbb P^1$}\label{sec:HF}
In this section we approach the study of the Hilbert function of these varieties. We start from the following specific case.

\begin{definition}
If $X$ is a variety of lines such that the index sets $U_1(X), U_2(X)$ and $ U_3(X)$ are Ferrers diagram, then we call $X$ a \textit{Ferrers variety of lines}.
That is, after renaming, we assume that if $\mathcal L(A_{i},B_{j})\in U_h(X)$  then  $\mathcal L(A_{i'},B_{j'})\in U_h(X)$ for every $1 \le i'\le i,\ 1 \le j'\le j$ and for each direction $h=1,2,3$.
\end{definition}

\begin{remark}
As a consequence of  Theorem \ref{thm:ch1}, note that a Ferrers variety of lines is ACM.
\end{remark}

 Recall that given a homogeneous ideal $I\subseteq R$, the \textit{Hilbert function of $R/I$} is the numerical function $$H_{R/I} : \mathbb{N}^3 \rightarrow \mathbb{N}$$ defined by $$ H_{R/I}(i,j,k) := dim_K (R/I)_{i,j,k} = dim_K R_{i,j,k} - dim_K I_{i,j,k}.$$
The \textit{first difference function of $H$}, denoted $\Delta H$, is the function $\Delta H : \mathbb{N}^3 \rightarrow \mathbb{N}$ defined by
$$ \Delta H(i,j,k) := \sum \limits_{(0,0,0) \le (l,m,n) \le (1,1,1)} (-1)^{l+m+n} H(i-l,j-m,k-n).$$

Now, let $X$ be a Ferrers variety of lines and let $X_3=\bigcup\limits_{(r,s)\in U_3(X)}\mathcal{L}(A_r,B_s)$ be the variety of lines consisting of the lines of $X$ of type $(1,1,0)$. Since $U_3(X)$ is a Ferrers diagram, the variety $X_3$ is ACM (in $\mathbb P^1\times \mathbb P^1$) and we can explicitly write out a set of minimal generators of $I_{X_3}$, see Remark \ref{r.cono} and \cite{GV-book}.  If $ \{ (a_{3,i}, b_{3,i})\} $ is the set of the degrees of these minimal generators, we denote by $D_3(X) := \{ (a_{3,i}, b_{3,i},0)\}$.

Analogously, if we consider the varieties of lines $X_1$ and $X_2$ consisting of the lines of $X$ of types $(0,1,1)$ and $(1,0,1)$, respectively, we obtain the sets of degrees $D_1(X) = \{ (0,b_{1,j}, c_{1,j})\} $ and $D_2(X) = \{ (a_{2,k}, 0, c_{2,k})\} $. Then we denote by

\begin{equation}
\begin{split}
D(X):=  \{ & (\max\{a_{3,i}, a_{2,k}\}, \max\{b_{3,i}, b_{1,j}\}, \max\{c_{1,j}, c_{2,k}\}  )\ |\ \forall \ (a_{3,i},b_{3,i},0)\in D_3(X), \\ \phantom{=} & (a_{2,k},0,c_{2,k})\in D_2(X), (0,b_{1,j},c_{1,j})\in D_1(X) \}. \nonumber
\end{split}
\end{equation}

\noindent Finally, we denote by $\hat{D}(X)$ the set of the minimal elements of $D(X)$ with respect to the natural partial order $\preceq$ on the elements of $\mathbb{N}^3$.

\begin{theorem}\label{beta_0}
    Let $X$ be a Ferrers variety of lines.  Then $I_X$ is minimally generated by the following set of forms
    \[\left\{ \prod_{i\le a}A_i\prod_{j\le b}B_j\prod_{k\le c}C_k  \ |\ \text{for each}\ (a,b,c)\in \hat{D}(X)   \right \}.
    \]

\end{theorem}
\begin{proof} First, we prove that if $ (a,b,c)\in \hat{D}(X)$, then $\prod_{i\le a}A_i\prod_{j\le b}B_j\prod_{k\le c}C_k\in I_X.$
Indeed $(a,b,c)\in \hat{D}(X)$ implies $\prod_{i\le a}A_i\prod_{j\le b}B_j \in I_{X_3}$, $\prod_{j\le b}B_j\prod_{k\le c}C_k \in I_{X_1}$, $\prod_{i\le a}A_i\prod_{k\le c}C_k \in I_{X_2}$ and they are not necessarily minimal elements of the respective ideal. Thus $\prod_{i \le a} A_i\prod_{j \le b} B_j\prod_{k \le c} C_k\in I_{X_1}\cap I_{X_2}\cap I_{X_3} = I_X$.
 Now, we show that if $ (a,b,c)\in \hat{D}(X)$ and $a>0$, then $\prod_{i\le a-1}A_i\prod_{j\le b}B_j\prod_{k\le c}C_k\notin I_X.$ This fact follows by contradiction. Indeed if  $\prod_{i\le a-1}A_i\prod_{j\le b}B_j\prod_{k\le c}C_k\in I_X,$ then $(a-1,b,0), (a-1,0,c), (0,b,c)$  are degrees of some (not necessarily minimal) elements in the ideal and therefore there is an element in $D(X)$ less than or equal to  $(a-1,b,c)$, contradicting  the minimality of $(a,b,c)\in \hat{D}(X)$.
Analogously, it can be easily showed that if $ (a,b,c)\in
\hat{D}(X)$ and $b>0$ (or $c>0$), then $\prod_{i\le
a}A_i\prod_{j\le b-1}B_j\prod_{k\le c}C_k\notin I_X$ (or
$\prod_{i\le a}A_i\prod_{j\le b}B_j\prod_{k\le c-1}C_k\notin
I_X$). Finally, we claim that $I_X $ is minimally generated  by
the forms $ \prod_{i\le a}A_i\prod_{j\le b}B_j\prod_{k\le c}C_k$
with $(a,b,c)\in \hat{D}(X)$. Take a form $F\in I_X$, without loss
of generality we can assume that $F:=\prod_{i\in \mathcal
A}A_i\prod_{j\in \mathcal B}B_j\prod_{k\in \mathcal C}C_k$ is
product of linear forms. By contradiction we assume $A_i$ divides $F$ and
$A_{i-1}$ does not divide $F$.     Then $\prod_{i\in \mathcal
A}A_i\prod_{j\in \mathcal B}B_j \in I_{X_3}$. Then $F\in
(\prod_{i\le a'}A_i\prod_{j\le b'}B_j)$ for some $a',b'.$
Repeating the same argument with respect to the other two
directions we get the proof. The minimality come from the
minimality of the degrees in $\hat{D}(X)$.
\end{proof}

The following corollary is an immediate consequence of Theorem \ref{beta_0} and the ACM property.
Set $\langle \hat{D}(X) \rangle := \{(i,j,k)\ |\ (i,j,k) \ge (a,b,c),\ \text{for some}\ (a,b,c) \in \hat{D}(X)  \}$.
\begin{corollary}\label{GORFerrers} Let $X$ be a Ferrers variety of lines. Then
        \[
        \Delta H_X{(i,j,k)}=\begin{cases}
        0 & \text{if}\ (i,j,k) \in \langle \hat{D}(X) \rangle \\
        1 & \text{otherwise}
        \end{cases}.
        \]
\end{corollary}

\begin{example}
    Let us consider the following variety of lines
    $$X=\{\mathcal{L} (A_i, B_j) \cup \mathcal{L}(A_i,C_k) \cup \mathcal{L}(B_j,C_k) \ | \  1\le i \le 4, \ 1\le j \le 3, \ 1\le k \le 2 \}.$$
    In this case we have $D_3(X)=\{(4,0,0), (0,3,0) \}$, $D_2(X)=\{(4,0,0), (0,0,2) \}$ and  $D_1(X)=\{(0,3,0), (0,0,2) \}$.
    Then $D(X)=\{(4,3,2), (4,3,0), (4,0,2), (0,3,2) \}$ and  $\hat{D}(X)=\{(4,3,0), (4,0,2), (0,3,2) \}$.
    Therefore, from Theorem \ref{beta_0}, a minimal set of generators of $I_X$ is given by: $$A_1A_2A_3A_4B_1B_2B_3, \ A_1A_2A_3A_4C_1C_2, \  B_1B_2B_3C_1C_2 $$ and
    \[
    \Delta H_X{(i,j,k)}=\begin{cases}
    0 & \text{if}\ (i,j,k) \ge (4,3,0) \ \text{or} \ (4,0,2) \ \text{or} \ (0,3,2) \\
    1 & \text{otherwise}
    \end{cases}.
    \]
\end{example}

\section{Case study: grids of lines and complete intersections of lines}\label{sec:grids}

In the last section we focus on the study of special arrangements of lines  in $\mathbb{P}^1 \times \mathbb{P}^1 \times \mathbb{P}^1$ having the ACM property. Recall that for a point $P\in \mathbb{P}^1 \times \mathbb{P}^1 \times \mathbb{P}^1$ there are exactly three lines passing through $P$, one for each direction. We have the following definition.

\begin{definition}
    Let $\mathcal Y$ be a finite set of points in $\mathbb{P}^1 \times \mathbb{P}^1 \times \mathbb{P}^1$. We call \textit{grid of lines} arising from $\mathcal Y$, and denote it by $X_{\mathcal Y}$, the set containing all the lines of $\mathbb{P}^1 \times \mathbb{P}^1 \times \mathbb{P}^1$ passing through some point of $\mathcal Y$.
\end{definition}

In other words, if $\mathcal Y$ is a finite set of points in $\mathbb{P}^1 \times \mathbb{P}^1 \times \mathbb{P}^1$, 
then $$X_{\mathcal Y}:= \bigcup_{P_{ijk} \in   \mathcal Y }   \mathcal{L} (A_i, B_j) \cup \mathcal{L}(A_i,C_k) \cup  \mathcal{L}(B_j,C_k) $$
where  $P_{ijk}:=\mathcal{L}(A_i)\cap\mathcal{L}(B_j)\cap\mathcal{L}(C_k).$

The next example shows that, even if  $\mathcal Y$ is an ACM set of points,  $X_{\mathcal Y}$ could be not ACM.

\begin{example}
    Suppose $\mathcal Y:=\{P_{112},P_{122}, P_{121}, P_{212}\} \subseteq \mathbb{P}^1 \times \mathbb{P}^1 \times \mathbb{P}^1.$ According to \cite{FGM}, $\mathcal Y$ is an ACM set of points. We have $\mathcal{L} (A_2, B_1),\mathcal{L}(B_2,C_1) \in X_{\mathcal Y}$ and $\mathcal L(A_2, B_2),\mathcal{L}(A_2,C_1)$, $\mathcal{L}(B_1,C_1) \notin X_{\mathcal Y} $, that is, $X_{\mathcal Y} $ has not the $Hyp_4(\star)$-property  and then $X_{\mathcal Y}$ is not ACM.
\end{example}

It is interesting to ask which sets of points $\mathcal Y\subseteq \mathbb{P}^1 \times \mathbb{P}^1 \times \mathbb{P}^1$ lead to an ACM grid of lines $X_{\mathcal Y}.$
A special class of CM rings is  represented by complete
intersections. We recall their definitions and properties.

\begin{definition}\label{CI}
An ideal $I\subset  R$ is a \textit{complete intersection} if it is generated by a
regular sequence.
\end{definition}

As pointed out in  \cite{GV-book}, Lemma 2.25 a complete
intersection is also Cohen-Macaulay.

\begin{definition}\label{d.CI}
     In $ \mathbb{P}^1 \times \mathbb{P}^1 \times \mathbb{P}^1$, we say that a set of points  $\mathcal C$ is a \textit{complete intersection of points of type $(a_1,a_2,a_3)$} if $I_{\mathcal C} = (F_1, F_2, F_3)$ is a complete intersection and $\deg F_i = a_i \bf e_i$  for $i=1,2,3$.

\end{definition}
Note that each $F_i$ in Definition \ref{d.CI} is product of linear forms.

\begin{definition}
    We say that a variety of lines $X$ is a \textit{ complete intersection of lines} in $\mathbb{P}^1 \times \mathbb{P}^1 \times \mathbb{P}^1$ if $I_X$ is a complete intersection.
\end{definition}

\begin{theorem} \label {thmACM} Let $\mathcal C\subseteq \mathbb{P}^1 \times \mathbb{P}^1 \times \mathbb{P}^1$ be a complete intersection of points of type $(a,b,c).$ Then $X_{\mathcal C} $ is ACM and a trigraded minimal free resolution of $I_{X_{\mathcal C}}$ is
    $$0 \rightarrow  R^2(-a,-b,-c) \rightarrow R(-a,-b,0) \oplus R(-a,0,-c) \oplus R(0,-b,-c) \rightarrow I_{X_{\mathcal C}} \rightarrow 0.$$
\end{theorem}
\begin{proof}
    The grid of lines $X:=X_{\mathcal C}$ has the $Hyp_n(\star)$-property for $n=4, 5, 6$ and then, by Theorem \ref{thm:ch1}, $X$ is ACM. Moreover, by Corollary \ref{c.monomial} the generators of $I_X$ are product of linear forms, so
    $$I_X= \left(\prod \limits _{\substack{ i \in [a] }} A_i \prod \limits_{\substack{j \in [b]}} B_j \ , \ \prod \limits _{\substack{i \in [a]}} A_i \prod \limits_{\substack{k \in [c]}} C_k \ , \ \prod \limits _{\substack{j \in [b]}} B_j \prod \limits_{\substack{k \in [c]}} C_k \right).$$
    Then a Hilbert-Burch matrix of $I_X$ is
    \begin{displaymath} \left( \begin{matrix} \prod \limits _{\substack{ i \in [a]}} A_i & \prod \limits _{\substack{ i \in [a]}} A_i \\ \prod \limits_{\substack{j \in [b]}} B_j & 0 \\ 0 & \prod \limits_{\substack{k \in [c]}} C_k
    \end{matrix} \right).
    \end{displaymath}
\end{proof}

\begin{example}
If $\mathcal C \subseteq \mathbb{P}^1 \times \mathbb{P}^1 \times
\mathbb{P}^1$ is a complete intersection of points of type
$(2,3,2)$, then the grid $X_{\mathcal C}$ is formed by 6 lines of
type $(1,1,0)$, 4 lines of type $(1,0,1)$ and 6 lines of type
$(0,1,1)$:
    \begin{center}
        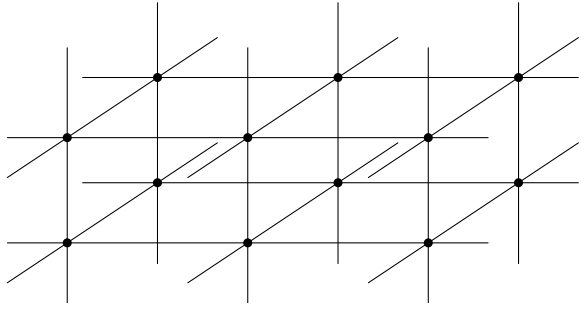
\begin{figure}[H]
            \centering
            \begin{tikzpicture}[scale=0.4]





            \draw (2,0) - - (2,8.5);
            \draw (5,1.3) - - (5,10);
            \draw (8,0) - - (8,8.5);
            \draw (11,1.3) - - (11,10);

            \draw (14,0) - - (14,8.5);
            \draw (17,1.3) - - (17,10);

            \draw (0,2) - - (16,2);
            \draw (2.5,4) - - (19,4);
            \draw(0,55/10) - - (16,55/10);
            \draw(2.5,75/10) - - (19,75/10);

            \draw (0,2/3) - - (7,16/3);
            \draw (0,25/6) - - (7,53/6);
            \draw (6,2/3) - - (13,16/3);
            \draw (6,25/6) - - (13,53/6);

            \draw (12,2/3) - - (19, 16/3);
            \draw (12,25/6) - - (19, 53/6);

            \fill (2,2) circle (1.5mm);
            \fill (8,2) circle (1.5mm);
            \fill (14,2) circle (1.5mm);
            \fill (8,5.5) circle (1.5mm);
            \fill (2,5.5) circle (1.5mm);
            \fill (14, 5.5) circle (1.5mm);
            \fill (5,4) circle (1.5mm);
            \fill (11,4) circle (1.5mm);
            \fill (17,4) circle (1.5mm);
            \fill (11,7.5) circle (1.5mm);
            \fill (5,7.5) circle (1.5mm);
            \fill (17, 7.5) circle (1.5mm);



            \end{tikzpicture}
            \caption{The grid of lines $X_{\mathcal C}$ arising from  a CI of type $(2,3,2)$. }
        \end{figure}
    \end{center}
In particular, $ I_{X_{\mathcal C}}$ has a trigraded minimal free  resolution of the following type
$$0 \rightarrow  R^2(-2,-3,-2) \rightarrow R(-2,-3,0) \oplus R(-2,0,-2) \oplus R(0,-3,-2) \rightarrow I_{X_{\mathcal C}} \rightarrow 0.$$
\end{example}

The following example shows that there exists an ACM grid of lines $X_{\mathcal Y}$ arising from a not ACM set of points $\mathcal Y$.
\begin{example}
    The following set of points $\mathcal Y:=\{ P_{111}, P_{121},P_{211}, P_{122},P_{212}, P_{222}\}$ is not an ACM set of points in  $\mathbb{P}^1 \times \mathbb{P}^1 \times \mathbb{P}^1$ (see \cite{FGM}). However, $X_{\mathcal Y}=X_{\mathcal C}$ where $\mathcal C:=\{ P_{ijk}\ |\  1\le i,j,k \le 2\}$, and then $X_{\mathcal Y}$ is an ACM grid of lines.
\end{example}

From Theorem \ref{thmACM}, we note that the ideal $I_{X_{\mathcal
C}}$  is generated by three forms that do not form a regular
sequence. That is, even if $\mathcal C$ is a complete intersection
of points, then its associated $X_{\mathcal C}$ variety of lines
is not a complete intersection of lines. Thus, it is natural to
study which varieties of lines are defined by a complete
intersection, i.e., their defining ideal has only two generators.
Theorem \ref{thmcompleteintersection} and Remark \ref{CIlines}
will describe complete intersections  of lines in $\mathbb{P}^1
\times \mathbb{P}^1 \times \mathbb{P}^1$.

\begin{remark}\label{rem:dgens}
    If $X$ is an ACM variety of lines, from Corollary  \ref{c.monomial}, $I_X$ is generated by products of linear forms. Then
    $$I_X\supseteq \left(\prod \limits _{\substack{ i \in [a]}} A_i \prod \limits_{\substack{j \in [b]}} B_j \ , \ \prod \limits _{\substack{i \in [a]}} A_i \prod \limits_{\substack{k \in [c]}} C_k \ , \ \prod \limits _{\substack{j \in [b]}} B_j \prod \limits_{\substack{k \in [c]}} C_k \right).$$
    So any set of minimal generators of $I_X$ contains one element of degree $(a_3, b_3, 0)$, one element of degree $(a_2,0,c_2)$ and one element of degree $(0,b_1,c_1)$.
\end{remark}

\begin{theorem} \label{thmcompleteintersection}
    Let $X$ be a variety of lines of $\mathbb{P}^1 \times \mathbb{P}^1 \times \mathbb{P}^1$. Then the ideal $I_X$ is a complete intersection if and only if  $I_X= (F_1,F_2)$, with $\deg F_1=a \textbf{e}_i$
    and $\deg F_2=b\textbf{e}_{j}+c\textbf{e}_{k}$ with $j,k \neq i$, for some  $a,b,c \in \mathbb{N}.$
\end{theorem}
\begin{proof}
    One implication is trivial. Let  $I_X$ be a complete intersection, i.e. $I_X$ is generated by a regular sequence of length 2, then $X$ is ACM. So, from Remark  \ref{rem:dgens}, any set of minimal generators of $I_X$ contains one element $G_1$ of degree $(0, b_1, c_1)$, one element $G_2$ of degree $(a_2, 0, c_2)$ and one element $G_3$ of degree $(a_3, b_3, 0)$ for some integers $a_i,b_j,c_k$. Since $I_X$ is a complete intersection, one of these three generators say, without loss of generality, the one of degree $(0, b_1, c_1)$, is not minimal, i.e. $G_1\in (G_2, G_3)$. This easily implies $a_2a_3=0.$
\end{proof}

\begin{remark}\label{CIlines}
From Theorem \ref{thmcompleteintersection}, a complete intersection  of lines $X$ is then obtained from a grid arising from a complete intersection of points by removing either all the lines having direction $\textbf{e}_i$ for some $i$,  or all the lines having direction $\textbf{e}_i$ and $\textbf{e}_j$ with $i\neq j$.
    Indeed, from Remark \ref{rem:dgens}, we have for instance $$I_X=\left(\prod \limits _{\substack{ i \in [a]}} A_i , \ \prod \limits _{\substack{j \in [b]}} B_j \prod \limits_{\substack{k \in [c]}} C_k \right)= \bigcap_{i \in [a] \atop  j \in [b]}(A_i,B_j) \cap \bigcap_{i \in [a] \atop  k \in [c]}(A_i,C_k).$$
\end{remark}

\begin{example}
    Let $X$ be the set of lines of $\mathbb{P}^1 \times \mathbb{P}^1 \times \mathbb{P}^1$ obtained by a grid of lines $X_{\mathcal C}$ arising from a complete intersection $\mathcal C$ of type $(4,3,2)$ removing all the lines having direction $\bf{e_2}$: $$X= \bigcup_{i \in [4] \atop j \in [3]} \mathcal{L} (A_i, B_j) \bigcup_{j \in [3] \atop k \in [2]} \mathcal{L}(B_j,C_k).$$
    Then the ideal $I_X$ is a complete intersection and it is generated by the regular sequence $F_1=B_1B_2B_3$ and $F_2=A_1A_2A_3A_4C_1C_2$ of degree $(0,3,0)$ and $(4,0,2)$, respectively.

\end{example}

We end the paper with two research topics that are still under our investigation.
\begin{enumerate}\label{GOR}

\item Guida, Orecchia and Ramella, in \cite{GOR}, studied the \textit{complete grids}
of lines in $\mathbb{P}^3$, whose defining ideal is the 1-lifting
ideal of a specific monomial ideal $J$ in a polynomial ring $S$ in
three variables. In particular, from Example 4.9 in \cite{GOR} and
Corollary \ref{GORFerrers}, we noted that the first difference of
the Hilbert function of the ideal $I_{X_{\mathcal C}}$ of a grid
of lines arising from a complete intersection of points of type
$(2,2,2)$ in $\mathbb{P}^1 \times \mathbb{P}^1 \times
\mathbb{P}^1$ in degree $(i,j,k)$ is equal to $1$ if and only if
$(i,j,k)$ belongs to  the order ideal $N(J) \subseteq
\mathbb{N}^3$ of the specific monomial ideal $J= (x_{1}^2x_2^2,
x_1^2x_3^2,x_2^2x_3^2)$ in $S$.

\item
Let us consider the ACM varieties of lines $X$ and the Ferrers variety of lines $X'$ as in Figure \ref{fig7} and Figure \ref{fig8}, respectively. We have that, for each $h=1,2,3$, $X_h$ and $X'_h$ have the same Hilbert functions. We also get  $H_X=H_{X'}$.

   \begin{figure}[H]
   \begin{minipage}{.5 \textwidth}
   \begin{tikzpicture}[scale=0.44]

   \draw [ultra thick] (5,1.3) - - (5,10);

   \draw [ultra thick] (2.5,4) - - (13,4);

   \draw[ultra thick] (6,2/3) - - (13,16/3);

   \fill (5,4) circle (2mm);
   \fill (11,4) circle (2mm);

\node [font=\footnotesize, text width=10mm] at (2.8,10) { $\bf{\mathcal{L}(A_1,B_1)}$};
\node [font=\footnotesize, text width=10mm] at (0.4,4) { $\bf{\mathcal{L}(A_1,C_1)}$};
\node [font=\footnotesize, text width=10mm] at (7.9,2/3) { $\bf{\mathcal{L}(B_2, C_1)}$};

   \end{tikzpicture} \caption{} \label{fig7}
   \end{minipage}%
   \begin{minipage}{.5 \textwidth}
   \begin{tikzpicture}[scale=0.44]

   \draw [ultra thick] (5,1.3) - - (5,10);

   \draw [ultra thick] (2.5,4) - - (13,4);

   \draw [ultra thick] (0,2/3) - - (7,16/3);

   \fill (5,4) circle (2mm);

\node [font=\footnotesize, text width=10mm] at (2.8,10) { $\bf{\mathcal{L}(A_1,B_1)}$};
\node [font=\footnotesize, text width=10mm] at (0.4,4) { $\bf{\mathcal{L}(A_1,C_1)}$};
\node [font=\footnotesize, text width=10mm] at (2,2/3) { $\bf{\mathcal{L}(B_1, C_1)}$};

   \end{tikzpicture} \caption{}  \label{fig8}
   \end{minipage}
   \end{figure}

According to many experimental computations using CoCoA,
\cite{cocoa}, we ask the following question:
\begin{question}\label{HF}
    Let $X$ be an ACM variety of lines and $X'$ be a Ferrers variety of lines such that,  for $h=1,2,3,$ $X_h$ and $X'_h$ have the same Hilbert functions. Is it true that $H_X=H_{X'}?$
\end{question}
\end{enumerate}

\end{document}